\documentclass[11pt]{article}
\usepackage{amsmath,graphicx,multicol}
\usepackage{amsfonts}
\usepackage{color}
\usepackage{subcaption}
\usepackage{bbm}
\usepackage{flushend}
\usepackage{amssymb}
\usepackage{algorithm}
\usepackage{algorithmic}
\usepackage{tabularx,booktabs}
\usepackage[utf8]{inputenc}
\usepackage[english]{babel}

\usepackage{amsthm}
\usepackage{bm}
\usepackage{epstopdf}
\usepackage{hyperref}
\usepackage{cleveref}
\usepackage{url}
\usepackage{cite}
\usepackage{mathtools}
\usepackage[toc,page]{appendix}

\newtheorem{definition}{Definition}[]

\newtheorem{proposition}{Proposition}[]
\newtheorem{theorem}{Theorem}[]

\newtheorem{lemma}[]{Lemma}

\DeclareMathOperator{\E}{\mathbb{E}}
\def\O{{\mathcal O}}
\def\A{{\mathcal A}}
\def\B{{\mathcal B}}
\def\R{{\mathbb R}}

\def\E{\mathbb{E}}
\def\R{\mathbb{R}}

\def\h{\mathbf{h}}

\def\O{\mathcal{O}}

\newcommand{\RNum}[1]{\uppercase\expandafter{\romannumeral #1\relax}}
\makeatletter
\renewcommand{\fnum@figure}{Fig.~\thefigure}
\makeatother

\title{\vspace{-8mm}\textbf{Optimization via First-Order Switching Methods: Skew-Symmetric Dynamics and Optimistic Discretization}}
\date{}
\vspace{-0.05in}
\author{Antesh Upadhyay, Sang Bin Moon, and Abolfazl Hashemi
\thanks{Authors are with the School of Electrical and Computer Engineering, Purdue University, West Lafayette, IN 47907, USA. This work was supported in part by the National Science Foundation (NSF) under Grant CNS-2313109.}
}
\begin{document}
\maketitle

\begin{abstract}
Large-scale constrained optimization problems are at the core of many tasks in control, signal processing, and machine learning. Notably, problems with functional constraints arise when, beyond a performance{\nobreakdash-}centric goal (e.g., minimizing the empirical loss), one desires to satisfy other requirements such as robustness, fairness, etc. A simple method for such problems, which remarkably achieves optimal rates for non-smooth, convex, strongly convex, and weakly convex functions under first-order oracle, is Switching Gradient Method (SGM): in each iteration depending on a predetermined constraint violation tolerance, use the gradient of objective or the constraint as the update vector. While the performance of SGM is well-understood for non-smooth functions and in fact matches its unconstrained counterpart, i.e., Gradient Descent (GD), less is formally established about its convergence properties under the smoothness of loss and constraint functions. In this work, we aim to fill this gap. First, we show that SGM may not benefit from faster rates under smoothness, in contrast to improved rates for GD under smoothness. By taking a continuous-time limit perspective, we show the issue is fundamental to SGM's dynamics and not an artifact of our analysis. Our continuous-time limit perspective further provides insights towards alleviating SGM's shortcomings. Notably, we show that leveraging the idea of optimism, a well-explored concept in variational inequalities and min-max optimization, could lead to faster methods. This perspective further enables designing a new class of ``soft'' switching methods, for which we further analyze their iteration complexity under mild assumptions.
\end{abstract}
\section{Introduction}\label{sec:intro}
Consider the canonical functional constrained problem
\begin{equation}\label{eq:mainproblem}
\tag{P}
    w^\ast \in \arg\min_{w\in \R^d} \;f(w) \qquad \text{such that }\quad g(w)\leq 0.
\end{equation}
Our goal will be to find an  $\epsilon$-solution as defined formally next.

\begin{definition}\label{def:solution}
A solution $\hat{w}\in \R^d$ is called an  $\epsilon$-solution for problem \eqref{eq:mainproblem} if 
    \begin{equation}\label{eq:sol-def}
    f(\hat{w})-f(w^\ast)\leq \epsilon, \qquad g(\hat{w})\leq \epsilon.
\end{equation}
\end{definition}
Functional constrained problems arise in numerous problems in signal processing and control \cite{scutari2016parallel,shi2020penalty,teo2021applied} as well as in machine learning, such as learning under robustness, safety, and fairness requirements \cite{huang2019stable,rigollet2011neyman,zafar2019fairness,huang2023oracle}.
This problem is extensively studied, and numerous elegant algorithms with provable guarantees under a variety of settings are developed \cite{bertsekas2014constrained}. While the majority of these methods are naturally primal-dual in nature \cite{chambolle2011first,zhang2022solving,boob2022stochastic,kim2024fast} or require sequential quadratic programming ideas \cite{nesterovintroductory,bertsekas2014constrained,scutari2016parallel2}, a remarkably simple and duality- and projection-free method due to Naum Shor \cite{shor2012minimization} achieves optimal rate of convergence under the black-box model \cite{nemirovskii1979complexity}, notably in the case where $f$ and $g$ are non-smooth. Called Switching Gradient Method (SGM), the method is a close cousin of the simplest optimization algorithm for unconstrained problems, namely Gradient Descent (GD), and performs the following iterations for some constraint violation tolerance $\epsilon>0$ and learning rate $\eta_t$:
\begin{equation}\label{eq:SGM}
\tag{SGM}
    w_{t+1} = w_t - \eta_t u_t, \qquad u_t = \begin{cases}
        \nabla f(w_t),\quad g(w_t)\leq \epsilon\\\nabla g(w_t),\quad g(w_t)> \epsilon.
    \end{cases}
\end{equation}

It is well known that after $T$ iterations, SGM can find an $\epsilon$-solution with $T=\O(\epsilon^{-2})$ when $f$ and $g$ are non-smooth convex functions \cite{shor2012minimization,nesterovintroductory,lan2020algorithms}.
Imposing strong convexity improves the iteration complexity to $T=\tilde{\O}(\epsilon^{-1})$.
Achieving these \emph{optimal rates} under the first-order oracle by such a simple method is attractive. Indeed, numerous extensions of SGM are explored, e.g. extension to the non-Euclidean case (mirror-descent version) \cite{bayandina2018mirror}, stochastic setting \cite{lan2020algorithms,liu2025single}, SGM with adaptive learning rates \cite{bayandina2018mirror}, methods for safe reinforcement learning \cite{xu2021crpo}, to name a few. Furthermore, motivated by learning under fairness constraints, \cite{huang2023oracle} has recently demonstrated the optimality of SGM for non-smooth non-convex problems where $f$ and $g$ satisfy the notion of weak convexity \cite{davis2019stochastic}.

\subsection{Contribution}\label{sec:contribution}
A close examination of the aforementioned results reveals that: (i) SGM's analyses naturally build upon those for GD under corresponding non-smooth settings, e.g., \cite{polyak1987introduction}, and (ii) SGM's iteration complexity's dependence on $\epsilon$ is identical to GD's for unconstrained problems under similar non-smooth and (strong/weak) convexity assumptions on $f$. These observations motivated us to explore the following simple question:
\begin{center}
    \it Q1: Does SGM, like GD, benefit from smoothness of $f$ and $g$ and achieve faster rates?
\end{center}
Not much is formally established about the convergence properties of SGM under smoothness, even in the simplest scenario where $f$ and $g$ are convex and $L$-smooth.
This paper aims to fill this gap and provide a principled foundation for further studies of first-order switching methods. 

In our pursuit towards answering the above question, we demonstrate that when $f$ and $g$ are convex and $L$-smooth, a simple adaptation of GD's analysis still results in the iteration complexity of $\O(\epsilon^{-2})$ in contrast to GD's  $\O(\epsilon^{-1})$.
To understand the source of this shortcoming, i.e., the inability to benefit from smoothness, we study SGM's dynamics from two perspectives: (i) viewing SGM as GD for an unconstrained problem with a non-differentiable penalty, and (ii) the continuous-time limit of a continuous interpolation of SGM's dynamics. Notably, these perspectives help us to identify the skew-symmetric and discontinuous nature of SGM's dynamics as the main factors preventing it from achieving faster rates. These results suggest that the iteration complexity of $\O(\epsilon^{-2})$ is in fact inherent, motivating us to explore the following question:
\begin{center}
    \it Q2: Are there simple, duality-free switching methods that could potentially benefit from the smoothness of $f$ and $g$?
\end{center}
In our pursuit to answer this question, our continuous-time limit perspective helps us to identify a connection to the literature on min-max optimization and variational inequalities. Much like SGM, the continuous-time dynamics underlying these problems are also generally skew-symmetric in nature, serving as the main reason for the lack of convergence of GD's counterpart (i.e., gradient descent ascent or GDA for min-max optimization tasks). The developments in variational inequality literature hinges crucially on the idea of optimism as a mechanism to correct the \emph{wrong direction} of the flow of skew-symmetric dynamics. 

Leveraging such a connection, we develop representative algorithms and establish their guarantees in several settings. Additionally, our investigations lead to a novel framework with  ``soft'' switching mechanisms to potentially design new duality-free first-order methods that use merely vector additions and leverage them in diverse settings where variants of \eqref{eq:mainproblem} play a crucial role. Notably, this feature of the new methods make them attractive for large- and huge-scale problems \cite{nesterov2014subgradient}.

\subsection{Related Work}\label{sec:related}
As our goal in this paper is to develop an understanding of switching methods, we next review the literature in this direction and refer the reader to
\cite{bertsekas2014constrained,chambolle2011first,scutari2016parallel2} and related follow-up work for alternative approaches.

Reference \cite{alkousa2020modification} introduces a variant of the adaptive mirror descent method integrated in SGM for convex and non-smooth problems. By selectively evaluating stochastic sub-gradients of only violated constraints, the modified scheme reduces per-iteration cost while retaining the optimal complexity.
In \cite{stonyakin2019adaptive,stonyakin2019some}, the authors develop adaptive SGM methods that automatically adjust to the unknown Lipschitz constants for both the convex and strongly convex cases.
SGM is further extended and studied in \cite{stonyakin2020mirror}, which introduces novel step-size rules and adaptive stopping criteria. Extensions under an inexact model as well as in the online setting are further studied by \cite{titov2020analogues,titov2018mirror}. In \cite{bayandina2018mirror}, the authors develop adaptive mirror descent with switching sub-gradient rules and stopping criteria. While they study the case where $f$ could be smooth, in their analyses, they assume the non-smoothness of $g$.

More recently, \cite{lan2020algorithms} extended SGM to the stochastic setting and derived various iteration complexity results under convexity and strong convexity.
Further, \cite{huang2023oracle} extends the analysis of SGM to the non-convex setting where both $f$ and $g$ are weakly convex and establishes the optimal complexity of SGM for finding an $\epsilon$-nearly stationary point. Finally, concurrent to our work, \cite{liu2025single} studied a variant of SGM for stochastic problems with weakly convex loss and constraint functions. Their method is based on using a differentiable penalty function mirroring our discussion in Section \ref{sec:sgm-penality}, but the resultant method is different from those we explore in this paper.

\subsection{Organization}\label{sec:org}
The rest of the paper is organized as follows. Section \ref{sec:background} provides a brief overview of relevant concepts. Section \ref{sec:SGM} provides the analyses of SGM, particularly under the smoothness assumption. Section \ref{sec:soft} discusses the proposed technique to develop new first-order switching methods with a ``soft'' switching mechanism. Section \ref{sec:OSGM} proposes the idea of leveraging optimistic discretization and designing new methods. Section \ref{sec:exp} reports some numerical experiments verifying the theoretical results, while Section \ref{sec:conclusion} states some concluding remarks. Finally, the detailed proofs of the main theorems are presented in the appendices.

\section{Background}\label{sec:background}
We start by stating some notation.
Let $[T]:= \{1,\dots,T\}$ and $|\A|$ denote the cardinality of set $\A$. Furthermore, let $\mathbb{I}\{\cdot\}$ denote the standard (binary-valued) indicator function, i.e.,
\begin{equation}\label{eq:indicator}
   \mathbb{I}(x) =\begin{cases}
       1,\quad x>0\\ 0,\quad x\leq 0,
   \end{cases}
\end{equation}
and $[\cdot]_+ =\max\{0,\cdot\}$.
Next we review the necessary definitions.
\begin{definition}[Convexity, Lipschitzness, and Smoothness]
    For all $x,y \in \R^d$ a (differentiable) function is
    \begin{itemize}
        \item Convex if
        \begin{equation}
            f(y)\geq f(x)+\langle \nabla f(x),y-x \rangle.
        \end{equation}
         \item $G$-Lipschitz  for some $G>0$ if
           \begin{equation}
           |f(x)-f(y)|\leq  G\|x-y\| \quad \equiv  \quad  \|\nabla f(x)\|\leq G.
        \end{equation}
        \item $L$-smooth  for some $L>0$ if
                   \begin{equation}
           \|\nabla f(x)-\nabla f(y)\|\leq  L\|x-y\|.
        \end{equation}
    \end{itemize}
\end{definition}
We remark that the definitions of convexity and Lipschitzness stated above do not require differentiability, and one simply needs to replace $\nabla f(x)$ with a subgradient of $f$ at $x$, i.e. $u \in \partial f(x)$. In this paper, however, for simplicity of notation and derivations, we simply assume differentiability and use $\nabla f(x)$.

Finally, we remark on the generality of problem \eqref{eq:mainproblem} in terms of handling multiple constraints, including functional equality constraints.  If multiple constraints $g_i(w)\leq 0$ exist, one can let $g(w) = \max_i g_i(w)$. Furthermore, in the case of equality constraints $h(w) = 0$, we could equivalently and simultaneously impose $h(w)\leq 0$ and $-h(w)\leq 0$.

\section{Demystifying SGM's Iteration Complexity under Smoothness}\label{sec:SGM}
In this section, we analyze SGM under two settings and provide Theorems \ref{thm:sgm1} and \ref{thm:sgm2}. The first theorem aims to facilitate the ensuing analyses and further improve the existing results in certain aspects, e.g., improving constants and bringing the proof as close as possible to that of GD. The second theorem, to the best of our knowledge, is not formally stated in the literature. Afterwards, we provide two perspectives highlighting the inherent shortcomings of SGM when it comes to harnessing smoothness.

\subsection{SGM for Convex and Lipschitz Functions}\label{sec:sgm-nonsmooth}
We start by stating our first result regarding the convergence properties of SGM for convex and non-smooth problems. Note that while this setting is well studied, e.g., Chapter 3.2.4 in \cite{nesterovintroductory} and \cite{shor2012minimization,lan2020algorithms},  here we state a different version with a different proof for completeness.
\begin{theorem}\label{thm:sgm1}
    Consider the problem in \eqref{eq:mainproblem} and SGM in \eqref{eq:SGM}. Assume $f$ and $g$ are convex and $G$-Lipschitz. Define $D:=\|w_1-w^\ast\|$ and 
    \begin{equation}
       \A = \{t\in [T]\;| \;g(w_t)\leq \epsilon\},\qquad \bar{w} = \frac{1}{|\A|} \sum_{t\in \A} w_t.
    \end{equation}
            Then, if
                \begin{equation}
        \epsilon = \frac{DG}{\sqrt{T}},\qquad \eta_t = \eta =\frac{D}{G\sqrt{T}},
    \end{equation}
            it holds that $\A$ is nonempty, $\bar{w}$ is well-defined, and $\bar{w}$ is an $\epsilon$-solution for \eqref{eq:mainproblem}.
\end{theorem}
The above theorem establishes that SGM has an optimal iteration complexity of $\O(\epsilon^{-2})$. We also note that instead of establishing the convergence of the average iterate, one could select $t_s \in \A$ uniformly at random and, using a standard technique (see e.g., \cite{ghadimi2013stochastic}), show $w_{t_s}$ is an expected $\epsilon$-solution.
\subsection{SGM for Convex and Smooth Functions}\label{sec:sgm-smooth}
Next, we state our result regarding the convergence properties of SGM for convex and smooth problems. Note that this result seems to be stated formally for the first time.
\begin{theorem}\label{thm:sgm2}
    Consider the problem in \eqref{eq:mainproblem} and SGM in \eqref{eq:SGM}. Assume $f$ and $g$ are convex and $L$-smooth. Define $D:=\|w_1-w^\ast\|$ and the following
    \begin{equation}
    \begin{aligned}
                \Tilde{f}&:= \min_x f(x) >-\infty,\qquad\Tilde{g}:= \min_x g(x) >-\infty,\\\Delta_{\max} &:= \max\{f(w^\ast)-\Tilde{f},g(w^\ast)-\Tilde{g}\}\geq 0,
    \end{aligned}
    \end{equation}
    and
        \begin{equation}
       \A = \{t\in [T]\;| \;g(w_t)\leq \epsilon\},\qquad \bar{w} = \frac{1}{|\A|} \sum_{t\in \A} w_t.
    \end{equation} Then if 
            \begin{equation}
            \epsilon = \frac{2LD^2}{T}+\sqrt{\frac{8LD^2\Delta_{\max}}{T}},\quad \eta = \min\Big\{\frac{1}{2L}, \sqrt{\frac{D^2}{2L\Delta_{\max}T}}\big\},
        \end{equation}
        it holds that $\A$ is nonempty, $\bar{w}$ is well-defined, and $\bar{w}$ is an $\epsilon$-solution for \eqref{eq:mainproblem}.
        \end{theorem}
    The above theorem establishes that SGM has an iteration complexity of $\O(\epsilon^{-2})$, despite the smoothness of $f$ and $g$. This rate falls short of the complexity of GD for this setting, i.e., $\O(\epsilon^{-1})$ \cite{nesterovintroductory} unless $\Delta_{\max} = 0$. 
    
    A careful reader notices that the result of Theorem \ref{thm:sgm2} resembles the standard, unconstrained convergence results of Stochastic Gradient Descent (SGD) if we interpret $\Delta_{\max}$ as the variance of the stochastic oracle \cite{ghadimi2013stochastic}. The ``variance'' here, however, does not arise from a stochastic estimation of the gradients of $f$ and $g$; rather, it arises due to SGM's dynamics. Let us discuss this point somewhat informally next.

    In particular, we show that SGM can be thought of as SGD for the dual of problem \eqref{eq:mainproblem}, i.e.
    \begin{equation}
        \min_w \max_{\lambda \geq 0} f(w)+\lambda g(w).
    \end{equation}
Let $\lambda^\ast \ge 0$ be optimal dual variable. Suppose $\lambda^\ast$ is not a function of $w$ (which is not correct in general). Then, 
\begin{equation}\label{eq:reduced-dual}
 \min_w  f(w)+\lambda^\ast g(w) \quad \equiv \quad  \min_w \frac{1}{1+\lambda^\ast} f(w)+\frac{\lambda^\ast}{1+\lambda^\ast} g(w),
\end{equation}
and the update of GD for \eqref{eq:reduced-dual} is as follows
        \begin{equation}
        \begin{aligned}
            w_{t+1} = w_t -\eta \Bigg(\frac{1}{1+\lambda^\ast} \nabla f(w_t)+\frac{\lambda^\ast}{1+\lambda^\ast} \nabla g(w_t)\Bigg).
        \end{aligned}
    \end{equation}
    Let $z \in \{1,2\}$ be a random variable such that $\mathbb{P}(z=1) = \frac{1}{1+\lambda^\ast}$, and $\mathbb{P}(z=2) = \frac{\lambda^\ast}{1+\lambda^\ast}$. Define the stochastic oracle
    \begin{equation}
        \nabla\tilde{\ell}(w_t,z) = \begin{cases}
          \nabla f(w_t) \quad w.p. \quad \frac{1}{1+\lambda^\ast}\\
          \nabla g(w_t) \quad w.p. \quad \frac{\lambda^\ast}{1+\lambda^\ast}.
        \end{cases}
    \end{equation}
    Consequently, GD's update can be written equivalently as
            \begin{equation}
        \begin{aligned}
            w_{t+1} = w_t -\eta \E[\nabla\tilde{\ell}(w_t,z)].
        \end{aligned}
    \end{equation}
    It is now evident that SGM acts as a stochastic approximation of the above update: when $w_t$ is significantly infeasible, $\lambda^\ast$ is large such that it is more likely for the oracle to return $\nabla g(w_t) $ as the gradient estimator. Conversely, when $w_t$ is near-feasible, $\lambda^\ast$ is small, such that it is more likely for the oracle to return $\nabla f(w_t) $ as the gradient estimator.

    We again emphasize that the above discussion is informal and its main goal is to provide an intuitive way to interpret the result of Theorem \ref{thm:sgm2} in light of those for SGD. Much like SGD that with and without smoothness it achieves a complexity of $\O(\epsilon^{-2})$ for convex problems due to the detrimental impact of the variance, SGM cannot enjoy faster rates under smoothness due to the detrimental impact of $\Delta_{\max}$ either.

    Let us also remark on the assumption that $\tilde{f}$ and $\tilde{g}$ are bounded from below, as it seemingly excludes affine functions for which the minimum value is $-\infty$. Note for affine functions $L=0$ such that $L\times \Delta_{\max} = 0\times \infty $, which can be defined to be $0$. Indeed, there are no smooth and convex functions such that $\Delta_{\max}=\infty$ while $L\neq 0$.
\subsection[SGM as GD for Unconstrained Problems with Non-differentiable Penalty]%
{SGM as GD for Unconstrained Problems with\\ Non-differentiable Penalty}\label{sec:sgm-penality}
Next, we provide an argument suggesting the complexity of $\O(\epsilon^{-2})$ derived by Theorem \ref{thm:sgm2} is inherent.

Consider the following unconstrained problem
\begin{equation}\label{eq:surrogateproblem}
    \min_w f(w)+\rho[g(w)]_+,\qquad \rho >0.
\end{equation}
Evidently, for any $w\in\R^d$ which is feasible for \eqref{eq:mainproblem} it holds that $g(w)\leq 0$, implying the objective in \eqref{eq:surrogateproblem} becomes $f(w)$. Further, the two problems have the same solutions for large enough $\rho$ (see, e.g., Chapter 4.1 in \cite{bertsekas2014constrained}). Let us now consider the following modification
\begin{equation}\label{eq:surrogateproblemeps}
    \min_w f(w)+\rho[g(w)-\epsilon]_+,\qquad \rho,\epsilon >0.
\end{equation}
It is evident that when $g(w)\leq \epsilon$ the (sub)gradient is $\nabla f(w)$ while when $g(w)> \epsilon$ the (sub)gradient is $\nabla f(w)+\rho \nabla g(w)$. Thus, for large $\rho$ such that $\nabla f(w)+\rho \nabla g(w) \approx \rho \nabla g(w)$, GD for \eqref{eq:surrogateproblemeps}, i.e.
\begin{equation}\label{eq:SGMapprox}
\begin{aligned}
    w_{t+1} &= w_t - \eta_t u_t\\u_t&= \nabla f(w_t)+\mathbb{I}(g(w_t)-\epsilon)\rho \nabla g(w_t),
\end{aligned}
\end{equation}
approximates SGM's dynamics given in \eqref{eq:SGM}. 

As the $\max$ operator is non-differentiable, the complexity of GD for \eqref{eq:surrogateproblemeps} is generally $\O(\epsilon^{-2})$ under the first-order oracle despite the smoothness of $f$ and $g$. Thus, the complexity of SGM for \eqref{eq:mainproblem} is generally $\Omega(\epsilon^{-2})$ as well, suggesting the tightness of the result of Theorem \ref{thm:sgm2}. 

One might expect to achieve $\O(\epsilon^{-1})$ for \eqref{eq:surrogateproblemeps} by leveraging accelerated methods with complexity $\O(\sqrt{\Tilde{L}/\epsilon})$ \cite{nesterov1983method} along with smooth approximations such that $\Tilde{L} = \O(\text{poly}(\rho,d,\epsilon^{-1}))$ \cite{nesterov2005smooth,nesterov2017random}, thereby developing a simple approach to modify SGM and improve its complexity result. However, to ensure finding an $\epsilon$-solution one must set $\rho = \epsilon^{-\alpha}$ for some $\alpha>0$. Thus, this strategy may fail to attain rates faster than $\O(\epsilon^{-2})$.
\subsection{Continuous-Time Limit of SGM's Dynamics}\label{sec:sgm-dynamics}
We now provide an alternative view that further supports the tightness of the result of Theorem \ref{thm:sgm2} and could aid us in designing algorithms with possibly better rates.

Consider the following Ordinary Differential Inclusion (ODI)
\begin{equation}
\begin{aligned}
    \dot{w} &\in -\tilde{F}(w)\\
    \tilde{F}(w)&= \mathbb{I}(g(w)-\epsilon) \nabla g(w)+ (1-\mathbb{I}(g(w)-\epsilon)) \nabla f(w),
\end{aligned}
\end{equation}
which is readily seen to represent the continuous-time limit of SGM's dynamics in \eqref{eq:SGM}. Note that $\tilde{F}(w)$ is discontinuous in $w$. Thus, the conditions of Caratheodory's existence theorem \cite{o1997existence} are not met, and the ODI may not even have a solution near any initialization. This again highlights the unfavorable structure of SGM's dynamics. Accordingly, we instead consider the following Ordinary Differential Equation (ODE)
\begin{equation}\label{eq:approxODE}
\begin{aligned}
     \dot{w} &= - F(w)\\
     F(w)&=\sigma_\beta(g(w)-\epsilon) \nabla g(w)+ [1-\sigma_\beta(g(w)-\epsilon)]\nabla f(w)\\
     \sigma_\beta(x)&= \frac{1}{1+\exp(-\beta x)},
\end{aligned}
\end{equation}
for some smoothing or temperature parameter $\beta>0$. Note that $\sigma_\beta(x)$ is the sigmoid function.

Evidently, as $\beta\rightarrow \infty$, the dynamics in \eqref{eq:approxODE} approaches that of SGM's. The proposed ODE in \eqref{eq:approxODE} has potential benefits over the continuous-time limit of SGM's dynamics. Firstly, it is continuous in $w$; thus, by Caratheodory's existence theorem, it is guaranteed to have a solution near any initialization. Secondly, basic calculations show that under the assumption that $f$ and $g$ are smooth and Lipschitz, $F(w)$ will be a Lipschitz operator. Consequently, \eqref{eq:approxODE} is further guaranteed to have a unique solution by the Picard–Lindelöf theorem \cite{o1997existence}.

Next, we aim to study the properties of the Jacobian of the above flow in \eqref{eq:approxODE}, which is denoted by $J_{F}(w)$.
Jacobian for such an ODE describes how the flow locally curves, rotates, contracts, and expands. It is well-known that when the Jacobian is purely symmetric (as is the case with the gradient flow $\dot{w} = -\nabla f(w)$), the flow does not exhibit any rotational behavior due to the fact that the underlying physical system is dissipative \cite{muehlebach2021optimization}. On the other hand, when the Jacobian has a skew-symmetric component (e.g., the flow of bilinear min-max games or, equivalently, the harmonic oscillator $\ddot{w} = -w$), the flow may exhibit rotation as a result of conserving energy.  The next proposition establishes that despite the favorable (Lipchitz) continuity structure, the dynamics in \eqref{eq:approxODE} is inherently skew-symmetrical.
\begin{proposition}\label{propositionh:jacobian}
    Consider the ODE given by \eqref{eq:approxODE}. We have
    \begin{equation}
    \begin{aligned}
                J_{F}(w) &= \nabla^2 f(w)+\sigma_\beta(g(w)-\epsilon) \big(\nabla^2 g(w)-\nabla^2 f(w)\big)\\&\qquad+\sigma^\prime_\beta(g(w)-\epsilon) \big(\nabla g(w)-\nabla f(w)\big)(\nabla g(w))^\top.
    \end{aligned}
    \end{equation}
    Consequently, as long as $\nabla g(w)$ and $\nabla f(w)$ are linearly independent, $J_{F}(w)$ has a skew-symmetric component.
\end{proposition}
\begin{proof}
    The Jacobian can be readily found by applying the chain rule and matrix calculus.
    Clearly the only non-symmetric term is the outer product $\nabla f(w) (\nabla g(w))^\top$. Evidently, $ab^\top = ba^\top$ iff $a$ and $b$ are linearly dependent (to see this, multiply both sides by $a^\top$ and rearrange).
\end{proof}
Proposition \ref{propositionh:jacobian} establishes that the flow is generally skew-symmetrical and the skew-symmetric component gets stronger as $\beta \rightarrow \infty$ (corresponding to the derivative of $\sigma_\beta(\cdot)$) and $|g(w)| \rightarrow \epsilon$, i.e., as the proposed ODE in \eqref{eq:approxODE} approaches the continuous-time limit of SGM's dynamics and the solution approaches the boundary of the feasible set.

Consider an arbitrary flow $\dot{w} = - F(w)$ which by taking the integral can be written equivalently as
    \begin{equation}
	    w_{t+1} = w_t - \int_{t}^{t+1}  F(w_\tau) \cdot d\tau.
	\end{equation}
     The simple forward Euler discretization approximates \cite{butcher2016numerical}
        \begin{equation}
        \begin{aligned}
            	    w_{t+1} &= w_t - \int_{t}^{t+1} F(w_\tau) \cdot d\tau \\&\approx w_t - \int_{t}^{t+\eta} F(w_{t}) \cdot d\tau = w_t - \eta F(w_{t}).
        \end{aligned}
	\end{equation}
    Note that when $F(w) = \nabla f(w)$ forward Euler discretization update stated above corresponds to GD and is known to enjoy favorable convergence properties under the Lipschitzness of $\nabla f(w)$, i.e., the smoothness of $f$. Note, in this case, the flow is purely symmetric as the Jacobian of $\nabla f(w)$, i.e., the Hessian of $f(w)$, is a symmetric matrix. 

    For general flows where $F(w)$ is not necessarily a gradient and thus the flow is skew-symmetrical, it is very well known that forward Euler discretization is ineffective: The limiting solution does not converge to the equilibrium, and the ergodic average solution converges slowly.
    
    In our context, as we demonstrated in Proposition \ref{propositionh:jacobian},  the dynamics in \eqref{eq:approxODE} is inherently skew-symmetrical. Thus, the forward Euler discretization update of the continuous-time limit of SGM's dynamics, which readily is the original SGM in \eqref{eq:SGM}, will be an ineffective approach as well and incapable of last-iterate convergence and faster rates under smoothness. Thus, this discussion highlights that SGM may fail to attain a complexity better than $\O(\epsilon^{-2})$ due to the discontinuity of its flow and skew-symmetric nature of the flow's continuous approximations.
\section{Soft Switching Mechanism via Trimmed Hinge}\label{sec:soft}
Given the desired properties of the proposed dynamics in \eqref{eq:approxODE} discussed in Section \ref{sec:sgm-dynamics}, notably guaranteed existence of a solution and the uniqueness under certain Lipschitzness properties, a natural question is whether its immediate forward Euler discretization, i.e.
\begin{equation}\label{eq:soft-SGM}
\tag{SSGM}
    \begin{aligned}
        w_{t+1} &= w_t - \eta F(w_t)\\
        F(w)&=\sigma_\beta(g(w)-\epsilon) \nabla g(w)+ [1-\sigma_\beta(g(w)-\epsilon)]\nabla f(w),
    \end{aligned}
\end{equation}
which we refer to as Soft Switching Gradient Method (SSGM),
enjoys any convergence guarantees in the sense of finding an $\epsilon$-solution per Definition \ref{def:solution}. It turns out we can provide theoretical guarantees for SSGM if instead of the sigmoid-based approximation we adopt a ``trimmed-hinged'' approximation:
\begin{equation}
    \sigma_\beta(x)= \min\{1,[1+\beta x]_+\} = \text{Proj}_{[0,1]}(1+\beta x).
\end{equation}
In particular, we provide the following theorem.
\begin{theorem}\label{thm:sgm3}
    Consider the problem in \eqref{eq:mainproblem} and SSGM  in \eqref{eq:soft-SGM}. Assume $f$ and $g$ are convex and $G$-Lipschitz. Define $D:=\|w_1-w^\ast\|$ and 
    \begin{equation}
       \A = \{t\in [T]\;| \;g(w_t)< \epsilon\},\quad \bar{w} = \sum_{t\in \A} \alpha_t w_t,
    \end{equation}
    where 
    \begin{equation}
        \alpha_t = \frac{1-\sigma_\beta(g(w_t)-\epsilon)}{\sum_{t\in \A} [1-\sigma_\beta(g(w_t)-\epsilon)]}.
    \end{equation}
            Then, if
                \begin{equation}
        \epsilon = \frac{2DG}{\sqrt{T}},\qquad  \eta =\frac{D}{G\sqrt{T}},\quad \beta = \frac{2}{\epsilon},
    \end{equation}
            it holds that $\A$ is nonempty, $\bar{w}$ is well-defined, and $\bar{w}$ is an $\epsilon$-solution for \eqref{eq:mainproblem}.
\end{theorem}
The above theorem establishes that SSGM has an optimal iteration complexity of $\O(\epsilon^{-2})$. We also note that, crucially, the definition of set $\A$ is changed compared to the statements of Theorems \ref{thm:sgm1} and \ref{thm:sgm2}, and here we require strict inequality. Further, the solution is produced by finding a weighted average (and in the case of random sampling, a non-uniform sampling strategy). Note that for any $t\in \A$ we have $g(w_t)< \epsilon$ thus, by definition of $\beta(g(w_t)-\epsilon)$ in \eqref{eq:approxODE}, we have $\sigma_\beta(g(w_t)-\epsilon)<1$. Thus, a benefit of the weighted solution is that it puts more weight on iterates with smaller constraint values.

We remark that the dependence of $\beta$ on $\epsilon$ is potentially a very pessimistic estimate. The reason for that is that, as we elaborate in the proof of Theorem \ref{thm:sgm3}, we use the highly pessimistic bound $\sum_{t\in \A} \sigma_\beta(g(w_t)-\epsilon) <T$. In practice, however, $\beta$ can be treated as constant and independent of $T$. Indeed, in our experiments, we simply used $\beta=1$.

Let us conclude this section by estimating the convergence properties of SSGM when $f$ and $g$ are $L$-smooth. The proof, which is omitted for brevity, follows using ideas developed in the proofs of Theorems \ref{thm:sgm2} and \ref{thm:sgm3}.
\begin{theorem}\label{thm:sgm4}
    Consider the problem in \eqref{eq:mainproblem} and SSGM  in \eqref{eq:soft-SGM}. Assume $f$ and $g$ are convex and $L$-smooth. Define $D:=\|w_1-w^\ast\|$  and the following
    \begin{equation}
    \begin{aligned}
                \Tilde{f}&:= \min_x f(x) >-\infty,\qquad\Tilde{g}:= \min_x g(x) >-\infty,\\\Delta_{\max} &:= \max\{f(w^\ast)-\Tilde{f},g(w^\ast)-\Tilde{g}\}\geq 0,
    \end{aligned}
    \end{equation}
    and
        \begin{equation}
       \A = \{t\in [T]\;| \;g(w_t)< \epsilon\},\quad \bar{w} = \sum_{t\in \A} \alpha_t w_t,
    \end{equation}
    where 
    \begin{equation}
        \alpha_t = \frac{1-\sigma_\beta(g(w_t)-\epsilon)}{\sum_{t\in \A} [1-\sigma_\beta(g(w_t)-\epsilon)]}.
    \end{equation}
Then if 
            \begin{equation}
            \begin{aligned}
            \epsilon &= \frac{4LD^2}{T}+\sqrt{\frac{32LD^2\Delta_{\max}}{T}},\qquad \beta = \frac{2}{\epsilon},\\\eta &= \min\Big\{\frac{1}{2L}, \sqrt{\frac{D^2}{2L\Delta_{\max}T}}\big\},
            \end{aligned}
        \end{equation} 
            it holds that $\A$ is nonempty, $\bar{w}$ is well-defined, and $\bar{w}$ is an $\epsilon$-solution for \eqref{eq:mainproblem}.
\end{theorem}
\section{Optimistic Discretization}\label{sec:OSGM}
As we discussed in Sections \ref{sec:contribution} and \ref{sec:sgm-dynamics}, any differentiable approximation of the flow of SGM will be skew-symmetrical in nature. For such flows, it is very well known that forward Euler discretization is ineffective. A simple discretization that is proven effective, however, is the backward Euler method \cite{butcher2016numerical}
        \begin{equation}
        \begin{aligned}
            	    w_{t+1} &= w_t - \int_{t}^{t+1} F(w_\tau) \cdot d\tau \\&\approx w_t - \int_{t}^{t+\eta} F(w_{t+1}) \cdot d\tau = w_t - \eta F(w_{t+1}).
        \end{aligned}
	\end{equation}

We refer to this approach as optimistic discretization as it optimistically uses the direction specified by the predicted point to form the update.

It is then natural to explore the efficacy of this approach for the discretization of the SGM's flow as well as \eqref{eq:approxODE} and studying the convergence of the resultant methods.
In doing so, we first explore the following simple method
\begin{equation}\label{eq:SSPPM}
\tag{SSPPM}
    \begin{aligned}
        w_{t+1} &= w_t - \eta F(w_{t+1})\\
       F(w)&=\sigma_\beta(g(w)-\epsilon) \nabla g(w)+ [1-\sigma_\beta(g(w)-\epsilon)]\nabla f(w),
    \end{aligned}
\end{equation}
which we refer to it as Soft Switching Proximal Point Method (SSPPM). The reason we choose this name is the well-known connection of the celebrated proximal point method \cite{rockafellar1976monotone}
\begin{equation}\label{PPM}
    w_{t+1} = \arg\min_w \;\;f(w)+\frac{1}{2\eta}\|w-w_t\|^2,
\end{equation}
and backward Euler discretization when the flow is described by the gradient of a convex function $w_{t+1} = w_t - \eta \nabla f(w_{t+1})$. 

We state the following result for SSPPM.
\begin{theorem}\label{thm:sgm5}
    Consider the problem in \eqref{eq:mainproblem} and SSPPM  in \eqref{eq:SSPPM}. Assume $f$ and $g$ are convex. Define $D:=\|w_1-w^\ast\|$ and 
    \begin{equation}
       \A = \{t\in [T]\;| \;g(w_t)< \epsilon\},\quad \bar{w} = \sum_{t\in \A} \alpha_t w_t,
    \end{equation}
    where 
    \begin{equation}
        \alpha_t = \frac{1-\sigma_\beta(g(w_t)-\epsilon)}{\sum_{t\in \A} [1-\sigma_\beta(g(w_t)-\epsilon)]}.
    \end{equation}
            Suppose the nonlinear equation in \eqref{eq:SSPPM} has a solution for all $t$. Then if
                \begin{equation}
        \epsilon = \frac{D^2}{\eta T},\qquad \eta>0,\qquad  \beta = \frac{2}{\epsilon},
    \end{equation}
            it holds that $\A$ is nonempty, $\bar{w}$ is well-defined, and $\bar{w}$ is an $\epsilon$-solution for \eqref{eq:mainproblem}.
\end{theorem}
Note that the theorem does not require Lipschitzness or smoothness of $f$ and $g$. Yet, it suggests that SSPPM achieves the complexity of $\O(\epsilon^{-1})$ as long as the nonlinear equation in \eqref{eq:SSPPM} has a solution for all $t$. We will discuss this point in the remainder of this section. Another remark worth stating is the implicit nature of the update \eqref{eq:SSPPM}. We will discuss this aspect in the remainder of this section as well. Note that both of these considerations are generally applicable to the the original proximal point method for solving variational inequalities.
\subsection{Fixed Point Existence and  Calculation}
As stated, Theorem \ref{thm:sgm5} requires the existence of a solution for $ w= w_t - \eta F(w)$ for all $t$. At its core, this task is a fixed point problem. 

There exist many fixed-point theorems that could be leveraged to establish conditions under which a solution for \eqref{eq:SSPPM} exists. Notably, by the Krasnosel'ski\u{\i}–Mann Theorem \cite{krasnosel1955two,mann1953mean}, a sufficient condition is non-expansiveness of $\eta F(\cdot)$. Thus, as long as $\eta G_F\leq 1$ \eqref{eq:SSPPM} will have a solution where $G_F$ is the Lipschitz constant of $F(\cdot)$. Uniqueness is guaranteed if $\eta G_F< 1$ according to the Banach fixed-point theorem \cite{banach1922operations}. The next proposition establishes a bound on $G_F$ based on the properties $f$, $g$, and $\sigma_\beta(\cdot)$.
\begin{proposition}\label{prop:lip}
    Assume $f$ and $g$ are $G$-Lipschitz and $L$-smooth. Then, $G_F \leq 2(L+G^2 \beta)$.
\end{proposition}
Before utilizing the result of this proposition, let us remark that due to the discontinuity of $\mathbb{I}(\cdot)$ (and as a consequence the lack of Lipschitzness), the hard switching counterpart of SSPPM in \eqref{eq:SSPPM} is not guaranteed to have any solution.

Let us now use the result of Proposition \ref{prop:lip} along with the result of Theorem \ref{thm:sgm5}. Setting $\eta = 1/2(L+G^2 \beta)$ and equating $\epsilon =2/\beta$ and $\epsilon = D^2/\eta T$ leads to a quadratic equation for $\beta$. Solving this equation determines $\epsilon = DG/\sqrt{T}$ and $\beta = 2/\epsilon$.
This result implies that in the worst case and to guarantee a solution exists for all $t$, the complexity of SSPPM is $\O(\epsilon^{-2})$. Interestingly, while smoothness is leveraged in Proposition \ref{prop:lip}, the expression for the best $\epsilon$ and $\beta$ stated above are independent of $L$. Let us also state that we believe this bound is highly pessimistic due to the fact that we used a potentially pessimistic bound to establish $\beta = 2/\epsilon$ (please refer to the discussion following Theorem \ref{thm:sgm3}). Thus, treating $\beta$ as a constant which is independent of $T$ would still imply $\O(\epsilon^{-1})$  complexity.

Now, suppose a solution exists. The next concern is the implicit nature of the SSPPM update. As finding $w_{t+1}$ amounts to a fixed point calculation problem, many algorithms such as the classical fixed point iteration, Krasnosel'ski\u{\i}–Mann iteration, or their approximate and stochastic version can be leveraged \cite{liang2016convergence,hashemi2024unified,bravo2024stochastic}. If $\eta$ is set such that $\eta F(\cdot)$ is contractive, exact methods can generally converge linearly while stochastic ones may suffer from sublinear rates \cite{combettes2015stochastic}. While generally leveraging fixed point iteration methods is a promising approach, we pursue an alternative solution next.
\subsection{Explicit Optimistic Updates}
Here we explore a simple alternative by leveraging the connection between \eqref{PPM} and backward Euler discretization. As mentioned earlier, in our context, {$F(w)$} is not a gradient. However, we propose to simply replace the terms involving $\sigma_\beta(g(w_{t+1})-\epsilon)$ with $\sigma_\beta(g(w_t)-\epsilon)$ leading to Soft Switching Proximal Point Method-Explicit (SSPPM-E)
\begin{equation}\label{eq:SSPPM-E}
\tag{SSPPM-E}
    \begin{aligned}
        w_{t+1} &= \arg\min_w \;\; h_t(w)+\frac{1}{2\eta}\|w-w_t\|^2\\
       h_t(w)&= \sigma_\beta(g(w_t)-\epsilon) g(w)+ [1-\sigma_\beta(g(w_t)-\epsilon)]f(w).
    \end{aligned}
\end{equation}
In scenarios where $f$ and $g$ have closed-from proximal maps, e.g., $\ell_1$, $\ell_2$, affine, and quadratic functions, $h_t(w)$ will also have a closed-form proximal map. 

Next, we show that SSPPM-E enjoys the following optimal $\O(\epsilon^{-2})$ complexity guarantee.
\begin{theorem}\label{thm:sgm6}
    Consider the problem in \eqref{eq:mainproblem} and SSPPM-E  in \eqref{eq:SSPPM-E}. Assume $f$ and $g$ are convex and $G$-Lipschitz. Define $D:=\|w_1-w^\ast\|$ and 
    \begin{equation}
       \A = \{t\in [T]\;| \;g(w_t)< \epsilon\},\quad \bar{w} = \sum_{t\in \A} \alpha_t w_t,
    \end{equation}
    where 
    \begin{equation}
        \alpha_t = \frac{1-\sigma_\beta(g(w_t)-\epsilon)}{\sum_{t\in \A} [1-\sigma_\beta(g(w_t)-\epsilon)]}.
    \end{equation}
Then if
                \begin{equation}
        \epsilon = \frac{2\sqrt{2}DG}{\sqrt{T}},\qquad \eta = \frac{D}{G\sqrt{2T}},\qquad  \beta = \frac{2}{\epsilon},
    \end{equation}
            it holds that $\A$ is nonempty, $\bar{w}$ is well-defined, and $\bar{w}$ is an $\epsilon$-solution for \eqref{eq:mainproblem}.
\end{theorem}

Finally, we demonstrate that the connection between \eqref{PPM} and backward Euler discretization can be further leveraged to design a method using the original hard switching mechanism leveraged by SGM. We call the resultant scheme Switching Proximal Point Method (SPPM) with the following update:
\begin{equation}\label{eq:SPPM}
\tag{SPPM}
\begin{aligned}
        w_{t+1} &= \arg\min_{w} \; h_t(w)+\frac{1}{2\eta}\|w-w_t\|^2\\h_t(w) &= \begin{cases}
         f(w),\quad g(w_t)\leq \epsilon\\ g(w),\quad g(w_t)> \epsilon.
    \end{cases}
\end{aligned}
\end{equation}

Theorem \ref{thm:sgm7} below, whose proof is omitted for brevity as it relies simply on the ideas developed in the proofs of Theorems \ref{thm:sgm1} and \ref{thm:sgm6}, shows SPPM also enjoys an optimal $\O(\epsilon^{-2})$ complexity guarantee.
\begin{theorem}\label{thm:sgm7}
    Consider the problem in \eqref{eq:mainproblem} and SPPM  in \eqref{eq:SPPM}. Assume $f$ and $g$ are convex and $G$-Lipschitz. Define $D:=\|w_1-w^\ast\|$ and 
    \begin{equation}
       \A = \{t\in [T]\;| \;g(w_t)\leq \epsilon\},\qquad \bar{w} = \frac{1}{|\A|} \sum_{t\in \A} w_t.
    \end{equation}
Then if
                \begin{equation}
        \epsilon = \frac{\sqrt{2}DG}{\sqrt{T}},\qquad \eta = \frac{D}{G\sqrt{2T}},
    \end{equation}
            it holds that $\A$ is nonempty, $\bar{w}$ is well-defined, and $\bar{w}$ is an $\epsilon$-solution for \eqref{eq:mainproblem}.
\end{theorem}
\section{Verifying Experiments}\label{sec:exp}
In this section, we verify our theoretical results in terms of understanding the switching-based methods and their efficacy in solving \eqref{eq:mainproblem}. To this end, we consider a setup where $f$ and $g$ are random convex quadratic functions defined on $\R^{10}$. For all methods we use a constant step size that leads to the fastest convergence without facing any divergence issues, both in terms of loss value $f$ and infeasibility $[g]_+$. Further, we show the performance of the last iterate, i.e., all plots are the evaluation of $f(w_t)$ and $g(w_t)$.

\subsection{Observing the Skew-Symmetry Phenomenon}
First, we verify the result of Proposition \ref{propositionh:jacobian} regarding the inherent skew-symmetric properties of switching-based methods dynamics. To this end, we consider SSGM with the sigmoid smoothing function per \eqref{eq:approxODE} with different values of smoothing parameter $\beta$. Fig. \ref{fig:exact} (a) and (b) demonstrate the result of this experiment. Consistent with the result of Proposition \ref{propositionh:jacobian}, we observe that as the iterate approaches feasibility (around iteration 20), the dynamics becomes rotational in the senses that both $f$ and $g$ values oscillate around their respective equilibria. Furthermore, the rotation becomes more significant as $\beta$ increases. Both of these phenomena are consistent with Proposition \ref{propositionh:jacobian}.

\subsection{Comparison of Various Switching Methods}
Next, we compare various switching-based methods explored in this paper. In particular, we compare SGM \eqref{eq:SGM}, SSGM \eqref{eq:soft-SGM}, SPPM \eqref{eq:SPPM}, and SSPPM-E \eqref{eq:SSPPM-E} in terms of their loss and constraint values in Fig. \ref{fig:exact} (c) and (d). The figure demonstrates that SSGM and SSPPM-E enjoy faster and smoother convergence, in part thanks to the proposed soft switching mechanism. While SGM and SPPM further converge, their convergence is slower as they need a smaller stepsize to avoid divergence and large oscillation due to the hard switching mechanism employed by them. 

We further remark that SSGM and SSPPM-E enjoy favorable rates even when we use a small $\beta=1$ as the smoothing parameter. This is again consistent with our prior discussion on the fact that our bound on $\beta$ in Theorems \ref{thm:sgm3}--\ref{thm:sgm6} is pessimistic.

Finally, in Table \ref{table1} we estimate the empirical rates of convergence of these scheme for the setting explored in our experiments. To this end, we assume a power-law decay of the form $C t^{-\alpha}$ for the decay of loss and functions values $f$ and $g$ and estimate $\alpha$ by log-log regression. The table demonstrates that all methods seem to be significantly faster than the rates derived in Theorems \ref{thm:sgm1}--\ref{thm:sgm7}, particularly SSGM and SSPPM-E that leverage the proposed soft switching mechanism.
\begin{figure*}[t]
\centering
\begin{subfigure}[t]{0.24\textwidth}
    \includegraphics[width=\textwidth]{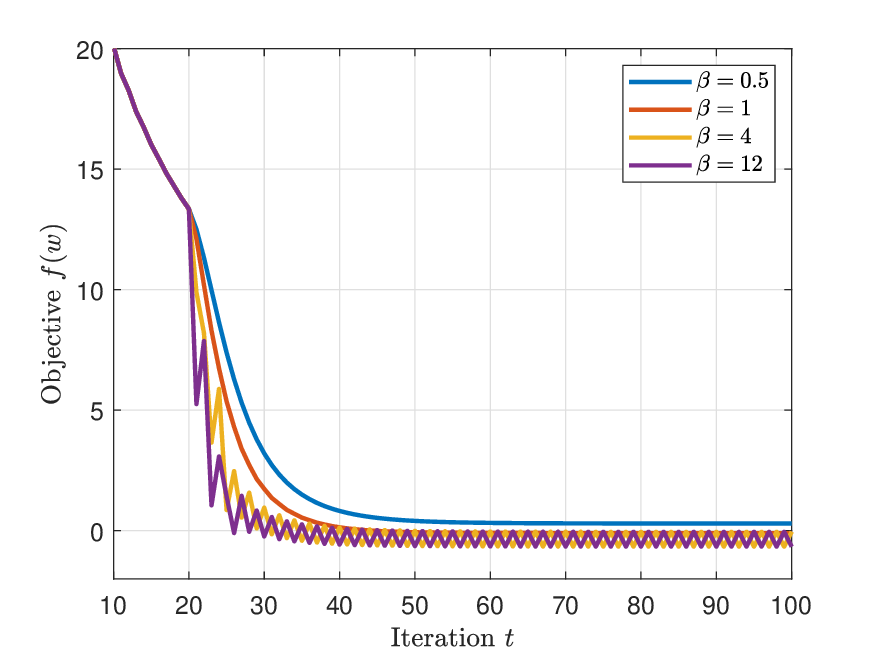}
    \caption{Skew Symmetric Dynamics:\\Impact of changing $\beta$ on $f$}
\end{subfigure}\hfill
\begin{subfigure}[t]{0.24\textwidth}
    \includegraphics[width=\textwidth]{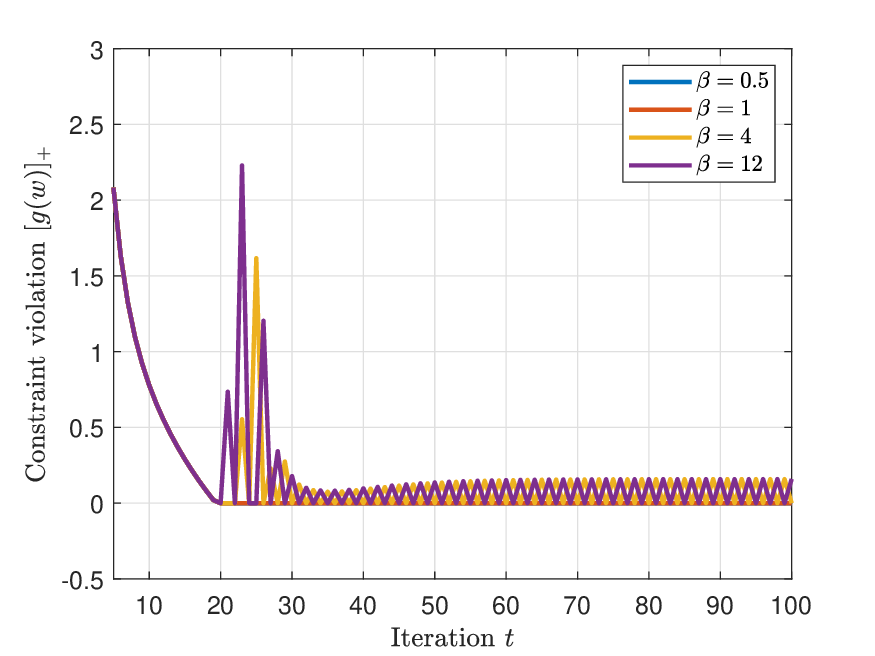}
    \caption{Skew Symmetric Dynamics:\\Impact of changing $\beta$ on $g$}
\end{subfigure}\hfill
\begin{subfigure}[t]{0.24\textwidth}
    \includegraphics[width=\textwidth]{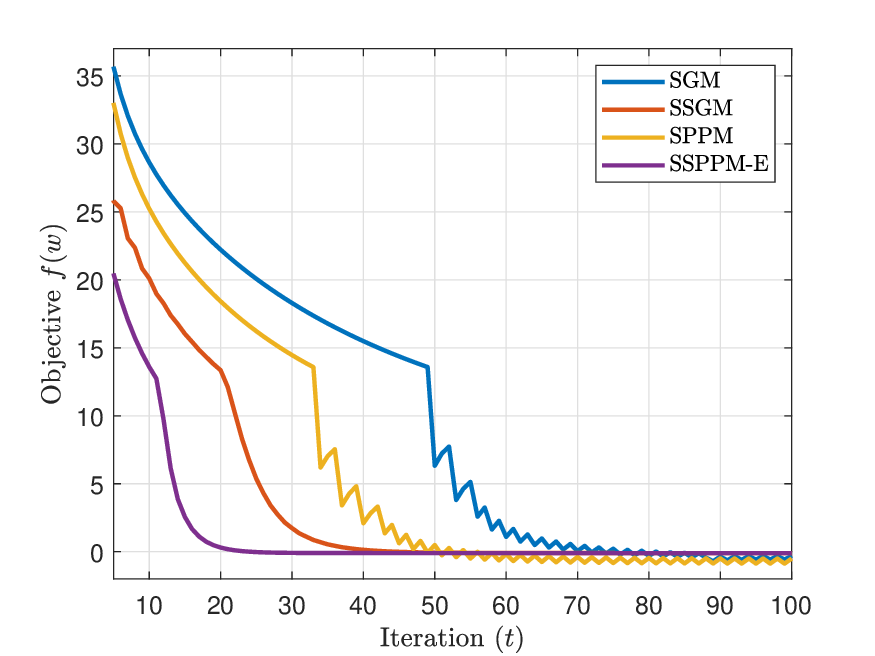}
    \caption{Convergence\\ Results for loss value $f$} 
\end{subfigure}\hfill
\begin{subfigure}[t]{0.24\textwidth}
    \includegraphics[width=\textwidth]{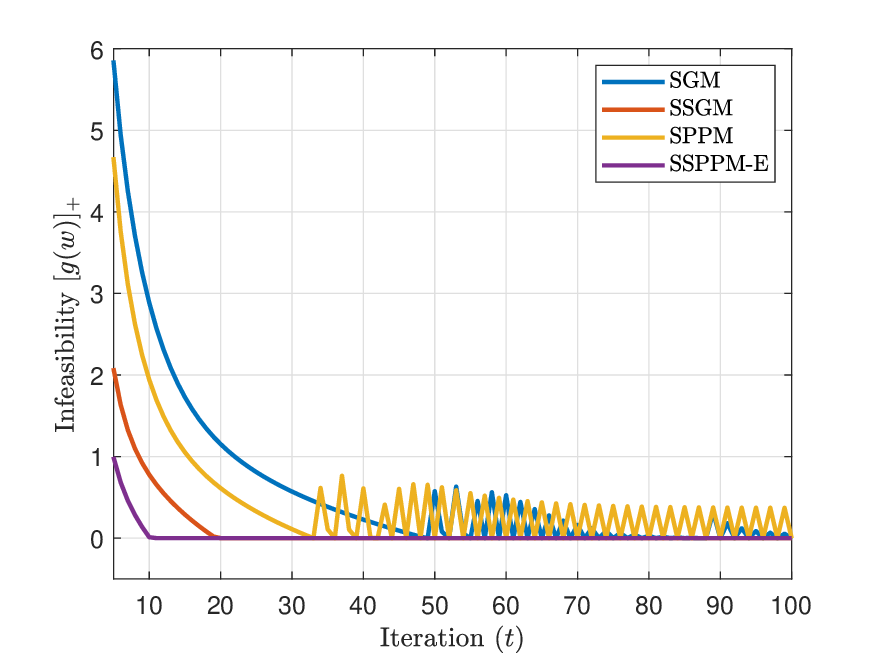}
    \caption{Convergence\\ Results constraint violation $[g]_+$} 
\end{subfigure}
\caption{Verifying experiments on an instance of \eqref{eq:mainproblem} with convex quadratic loss and constraint functions. \textbf{(a) and (b)} demonstrate the skew-symmetric nature of SGM and its smooth approximation \eqref{eq:approxODE} and its impact on values of $f$ and $g$. As the figure demonstrates, the skew-symmetric component becomes stronger as $|g(w)|$ approaches $\epsilon$ and $\beta$ increases, thereby verifying Proposition \ref{propositionh:jacobian}. \textbf{(c) and (d)} compare the loss and constraint values of various switching methods explored in this paper as a function of $t$, namely SGM \eqref{eq:SGM}, SSGM \eqref{eq:soft-SGM}, SPPM \eqref{eq:SPPM}, and SSPPM-E \eqref{eq:SSPPM-E}. For each method, we used the best constant step-size $\epsilon = 0.001$ and for SSGM and SSPPM-E we set $\beta=1$. The figure demonstrates SSGM and SSPPM-E, which incorporate the proposed soft mechanism, converge faster and smoother than their hard switching counterparts.}
\label{fig:exact}
\end{figure*}
\begin{table}[t]
    \centering
    \caption{Empirical rates of convergence of SGM \eqref{eq:SGM}, SSGM \eqref{eq:soft-SGM}, SPPM \eqref{eq:SPPM}, and SSPPM-E \eqref{eq:SSPPM-E} when assuming a power-law decay of the form $C t^{-\alpha}$. The table reports the $\alpha$ values.}
    \vspace{0.5em}
    \begin{tabular}{lcc}
        \toprule
        Algorithm & $\alpha$ for $f$ & $\alpha$ for $g$ \\
        \midrule
        SGM & 1.8 & 6.7 \\
        SSGM & 4.5 &  10\\
        SPPM & 2.5 &  6.3\\
        SSPPM-E & 6.8 & 7.4 \\
        \bottomrule
    \end{tabular}
    \label{table1}
\end{table}
\section{Conclusion}\label{sec:conclusion}
In this paper, we started by analyzing SGM in the smooth setting and established that it may not enjoy faster rates compared to the non-smooth setting. We attributed this result to the discontinuous and skew-symmetric dynamics of SGM. This perspective was then leveraged to design  Lipschitz dynamics by incorporating smoothing functions, notably the trimmed hinge. We then studied the convergence properties of the resulting methods when discretizing the dynamics by forward, backward, and approximate backward Euler discretizations. 

There are numerous interesting directions of research. Firstly, as we stated our bound for the smoothing parameter is highly pessimistic, and we conjecture it may be possible to derive a tighter bound resulting in a $\beta$ independent of $\epsilon$ and, equivalently, $T$.
Also, as the new methods seem superior to SGM, it would be of interest to extend them beyond the deterministic and convex setting studied in this paper by leveraging tools from \cite{asi2019stochastic,lan2020algorithms,huang2023oracle,liu2025single} which study extensions of SGM. Secondly, it is of interest to leverage the new methods in decision-making problems such as safe control and reinforcement learning \cite{xu2021crpo} particularly for cyber-physical systems, equilibrium computation and learning in games with functional constraints \cite{lauffer2023no,jordan2023first,jordan2024independent}, online learning and control with functional constraints \cite{rakhlin2013optimization,moonoptimistic}, and distributed optimization with functional constraints \cite{nedic2009distributed}. Finally, it is of interest to incorporate momentum mechanisms \cite{cutkosky2019momentum,tran2022hybrid,liu2020optimal,das2022faster} into the proposed methods to further improve their convergence properties.

\bibliographystyle{ieeetr}
\bibliography{refs}

\begin{thebibliography}{10}

\bibitem{scutari2016parallel}
G.~Scutari, F.~Facchinei, and L.~Lampariello, ``Parallel and distributed methods for constrained nonconvex optimization—part i: Theory,'' {\em IEEE Transactions on Signal Processing}, vol.~65, no.~8, pp.~1929--1944, 2016.

\bibitem{shi2020penalty}
Q.~Shi, M.~Hong, X.~Fu, and T.-H. Chang, ``{Penalty dual decomposition method for nonsmooth nonconvex optimization—Part II: Applications},'' {\em IEEE Transactions on Signal Processing}, vol.~68, pp.~4242--4257, 2020.

\bibitem{teo2021applied}
K.~L. Teo, B.~Li, C.~Yu, V.~Rehbock, {\em et~al.}, ``Applied and computational optimal control,'' {\em Optimization and Its Applications}, 2021.

\bibitem{huang2019stable}
L.~Huang and N.~Vishnoi, ``Stable and fair classification,'' in {\em International Conference on Machine Learning}, pp.~2879--2890, PMLR, 2019.

\bibitem{rigollet2011neyman}
P.~Rigollet and X.~Tong, ``Neyman-pearson classification, convexity and stochastic constraints,'' {\em Journal of Machine Learning Research}, 2011.

\bibitem{zafar2019fairness}
M.~B. Zafar, I.~Valera, M.~Gomez-Rodriguez, and K.~P. Gummadi, ``Fairness constraints: A flexible approach for fair classification,'' {\em The Journal of Machine Learning Research}, vol.~20, no.~1, pp.~2737--2778, 2019.

\bibitem{huang2023oracle}
Y.~Huang and Q.~Lin, ``Oracle complexity of single-loop switching subgradient methods for non-smooth weakly convex functional constrained optimization,'' in {\em Thirty-seventh Conference on Neural Information Processing Systems}, 2023.

\bibitem{bertsekas2014constrained}
D.~P. Bertsekas, {\em {Constrained optimization and Lagrange multiplier methods}}.
\newblock Academic press, 2014.

\bibitem{chambolle2011first}
A.~Chambolle and T.~Pock, ``A first-order primal-dual algorithm for convex problems with applications to imaging,'' {\em Journal of mathematical imaging and vision}, vol.~40, pp.~120--145, 2011.

\bibitem{zhang2022solving}
Z.~Zhang and G.~Lan, ``Solving convex smooth function constrained optimization is almost as easy as unconstrained optimization,'' {\em arXiv preprint arXiv:2210.05807}, 2022.

\bibitem{boob2022stochastic}
D.~Boob, Q.~Deng, and G.~Lan, ``Stochastic first-order methods for convex and nonconvex functional constrained optimization,'' {\em Mathematical Programming}, pp.~1--65, 2022.

\bibitem{kim2024fast}
J.~G. Kim, A.~Chandra, A.~Hashemi, and C.~Brinton, ``A fast single-loop primal-dual algorithm for non-convex functional constrained optimization,'' {\em arXiv preprint arXiv:2406.17107}, 2024.

\bibitem{nesterovintroductory}
Y.~Nesterov, {\em {Introductory Lectures on Convex Optimization: A Basic Course}}.
\newblock Springer, New York, 2004.

\bibitem{scutari2016parallel2}
G.~Scutari, F.~Facchinei, L.~Lampariello, S.~Sardellitti, and P.~Song, ``{Parallel and distributed methods for constrained nonconvex optimization-Part II: Applications in communications and machine learning},'' {\em IEEE Transactions on Signal Processing}, vol.~65, no.~8, pp.~1945--1960, 2016.

\bibitem{shor2012minimization}
N.~Z. Shor, {\em Minimization methods for non-differentiable functions}, vol.~3.
\newblock Springer Science \& Business Media, 2012.

\bibitem{nemirovskii1979complexity}
A.~Nemirovskii and D.~IUDIN, ``Complexity of problems and effectiveness of methods of optimization(russian book),'' {\em Moscow, Izdatel'stvo Nauka, 1979. 384}, 1979.

\bibitem{lan2020algorithms}
G.~Lan and Z.~Zhou, ``Algorithms for stochastic optimization with function or expectation constraints,'' {\em Computational Optimization and Applications}, vol.~76, no.~2, pp.~461--498, 2020.

\bibitem{bayandina2018mirror}
A.~Bayandina, P.~Dvurechensky, A.~Gasnikov, F.~Stonyakin, and A.~Titov, ``Mirror descent and convex optimization problems with non-smooth inequality constraints,'' {\em Large-scale and distributed optimization}, pp.~181--213, 2018.

\bibitem{liu2025single}
W.~Liu and Y.~Xu, ``A single-loop spider-type stochastic subgradient method for expectation-constrained nonconvex nonsmooth optimization,'' {\em arXiv preprint arXiv:2501.19214}, 2025.

\bibitem{xu2021crpo}
T.~Xu, Y.~Liang, and G.~Lan, ``Crpo: A new approach for safe reinforcement learning with convergence guarantee,'' in {\em International Conference on Machine Learning}, pp.~11480--11491, PMLR, 2021.

\bibitem{davis2019stochastic}
D.~Davis and D.~Drusvyatskiy, ``Stochastic model-based minimization of weakly convex functions,'' {\em SIAM Journal on Optimization}, vol.~29, no.~1, pp.~207--239, 2019.

\bibitem{polyak1987introduction}
B.~T. Polyak, ``Introduction to optimization. optimization software,'' {\em Inc., Publications Division, New York}, vol.~1, no.~32, p.~1, 1987.

\bibitem{nesterov2014subgradient}
Y.~Nesterov, ``Subgradient methods for huge-scale optimization problems,'' {\em Mathematical Programming}, vol.~146, no.~1, pp.~275--297, 2014.

\bibitem{alkousa2020modification}
M.~S. Alkousa, ``On modification of an adaptive stochastic mirror descent algorithm for convex optimization problems with functional constraints,'' {\em Computational Mathematics and Applications}, pp.~47--63, 2020.

\bibitem{stonyakin2019adaptive}
F.~S. Stonyakin, M.~Alkousa, A.~N. Stepanov, and A.~A. Titov, ``Adaptive mirror descent algorithms for convex and strongly convex optimization problems with functional constraints,'' {\em Journal of Applied and Industrial Mathematics}, vol.~13, no.~3, pp.~557--574, 2019.

\bibitem{stonyakin2019some}
F.~S. Stonyakin, M.~S. Alkousa, A.~A. Titov, and V.~V. Piskunova, ``On some methods for strongly convex optimization problems with one functional constraint,'' in {\em Mathematical Optimization Theory and Operations Research: 18th International Conference, MOTOR 2019, Ekaterinburg, Russia, July 8-12, 2019, Proceedings 18}, pp.~82--96, Springer, 2019.

\bibitem{stonyakin2020mirror}
F.~Stonyakin, A.~Stepanov, A.~Gasnikov, A.~Titov, {\em et~al.}, ``Mirror descent for constrained optimization problems with large subgradient values of functional constraints,'' {\em Computer research and modeling}, vol.~12, no.~2, pp.~301--317, 2020.

\bibitem{titov2020analogues}
A.~A. Titov, F.~S. Stonyakin, M.~S. Alkousa, S.~S. Ablaev, and A.~V. Gasnikov, ``Analogues of switching subgradient schemes for relatively lipschitz-continuous convex programming problems,'' in {\em International Conference on Mathematical Optimization Theory and Operations Research}, pp.~133--149, Springer, 2020.

\bibitem{titov2018mirror}
A.~A. Titov, F.~S. Stonyakin, A.~V. Gasnikov, and M.~S. Alkousa, ``Mirror descent and constrained online optimization problems,'' in {\em International Conference on Optimization and Applications}, pp.~64--78, Springer, 2018.

\bibitem{ghadimi2013stochastic}
S.~Ghadimi and G.~Lan, ``Stochastic first-and zeroth-order methods for nonconvex stochastic programming,'' {\em SIAM journal on optimization}, vol.~23, no.~4, pp.~2341--2368, 2013.

\bibitem{nesterov1983method}
Y.~Nesterov, ``A method for solving the convex programming problem with convergence rate o (1/k2),'' in {\em Dokl akad nauk Sssr}, vol.~269, p.~543, 1983.

\bibitem{nesterov2005smooth}
Y.~Nesterov, ``Smooth minimization of non-smooth functions,'' {\em Mathematical programming}, vol.~103, pp.~127--152, 2005.

\bibitem{nesterov2017random}
Y.~Nesterov and V.~Spokoiny, ``Random gradient-free minimization of convex functions,'' {\em Foundations of Computational Mathematics}, vol.~17, no.~2, pp.~527--566, 2017.

\bibitem{o1997existence}
D.~O'Regan, {\em Existence theory for nonlinear ordinary differential equations}, vol.~398.
\newblock Springer Science \& Business Media, 1997.

\bibitem{muehlebach2021optimization}
M.~Muehlebach and M.~I. Jordan, ``Optimization with momentum: Dynamical, control-theoretic, and symplectic perspectives,'' {\em Journal of Machine Learning Research}, vol.~22, no.~73, pp.~1--50, 2021.

\bibitem{butcher2016numerical}
J.~C. Butcher, {\em Numerical methods for ordinary differential equations}.
\newblock John Wiley \& Sons, 2016.

\bibitem{rockafellar1976monotone}
R.~T. Rockafellar, ``Monotone operators and the proximal point algorithm,'' {\em SIAM journal on control and optimization}, vol.~14, no.~5, pp.~877--898, 1976.

\bibitem{krasnosel1955two}
M.~A. Krasnosel'ski\u{\i}, ``Two remarks on the method of successive approximations,'' {\em Uspekhi matematicheskikh nauk}, vol.~10, no.~1, pp.~123--127, 1955.

\bibitem{mann1953mean}
W.~R. Mann, ``Mean value methods in iteration,'' {\em Proceedings of the American Mathematical Society}, vol.~4, no.~3, pp.~506--510, 1953.

\bibitem{banach1922operations}
S.~Banach, ``On the operations in abstract sets and their application to the equations,'' {\em Fundamenta mathematicae}, vol.~3, no.~1, pp.~133--181, 1922.

\bibitem{liang2016convergence}
J.~Liang, J.~Fadili, and G.~Peyr{\'e}, ``Convergence rates with inexact non-expansive operators,'' {\em Mathematical Programming}, vol.~159, pp.~403--434, 2016.

\bibitem{hashemi2024unified}
A.~Hashemi, ``A unified model for large-scale inexact fixed-point iteration: A stochastic optimization perspective,'' {\em IEEE Transactions on Automatic Control}, 2024.

\bibitem{bravo2024stochastic}
M.~Bravo and R.~Cominetti, ``Stochastic fixed-point iterations for nonexpansive maps: Convergence and error bounds,'' {\em SIAM Journal on Control and Optimization}, vol.~62, no.~1, pp.~191--219, 2024.

\bibitem{combettes2015stochastic}
P.~L. Combettes and J.-C. Pesquet, ``Stochastic quasi-fej{\'e}r block-coordinate fixed point iterations with random sweeping,'' {\em SIAM Journal on Optimization}, vol.~25, no.~2, pp.~1221--1248, 2015.

\bibitem{asi2019stochastic}
H.~Asi and J.~C. Duchi, ``Stochastic (approximate) proximal point methods: Convergence, optimality, and adaptivity,'' {\em SIAM Journal on Optimization}, vol.~29, no.~3, pp.~2257--2290, 2019.

\bibitem{lauffer2023no}
N.~Lauffer, M.~Ghasemi, A.~Hashemi, Y.~Savas, and U.~Topcu, ``No-regret learning in dynamic stackelberg games,'' {\em IEEE Transactions on Automatic Control}, vol.~69, no.~3, pp.~1418--1431, 2023.

\bibitem{jordan2023first}
M.~I. Jordan, T.~Lin, and M.~Zampetakis, ``First-order algorithms for nonlinear generalized nash equilibrium problems,'' {\em Journal of Machine Learning Research}, vol.~24, no.~38, pp.~1--46, 2023.

\bibitem{jordan2024independent}
P.~Jordan, A.~Barakat, and N.~He, ``Independent learning in constrained markov potential games,'' in {\em International Conference on Artificial Intelligence and Statistics}, pp.~4024--4032, PMLR, 2024.

\bibitem{rakhlin2013optimization}
S.~Rakhlin and K.~Sridharan, ``Optimization, learning, and games with predictable sequences,'' {\em Advances in Neural Information Processing Systems}, vol.~26, 2013.

\bibitem{moonoptimistic}
S.~B. Moon and A.~Hashemi, ``Optimistic regret bounds for online learning in adversarial markov decision processes,'' in {\em The 40th Conference on Uncertainty in Artificial Intelligence}, 2024.

\bibitem{nedic2009distributed}
A.~Nedic and A.~Ozdaglar, ``Distributed subgradient methods for multi-agent optimization,'' {\em IEEE Transactions on Automatic Control}, vol.~54, no.~1, pp.~48--61, 2009.

\bibitem{cutkosky2019momentum}
A.~Cutkosky and F.~Orabona, ``Momentum-based variance reduction in non-convex sgd,'' {\em Advances in neural information processing systems}, vol.~32, 2019.

\bibitem{tran2022hybrid}
Q.~Tran-Dinh, N.~H. Pham, D.~T. Phan, and L.~M. Nguyen, ``A hybrid stochastic optimization framework for composite nonconvex optimization,'' {\em Mathematical Programming}, vol.~191, no.~2, pp.~1005--1071, 2022.

\bibitem{liu2020optimal}
D.~Liu, L.~M. Nguyen, and Q.~Tran-Dinh, ``An optimal hybrid variance-reduced algorithm for stochastic composite nonconvex optimization,'' {\em arXiv preprint arXiv:2008.09055}, 2020.

\bibitem{das2022faster}
R.~Das, A.~Acharya, A.~Hashemi, S.~Sanghavi, I.~S. Dhillon, and U.~Topcu, ``Faster non-convex federated learning via global and local momentum,'' in {\em Uncertainty in Artificial Intelligence}, pp.~496--506, PMLR, 2022.

\end{thebibliography}
\appendix
\section*{Appendix}
\section{Proof of Theorem \ref{thm:sgm1}}
We first state a well-known result.
    \begin{lemma}\label{lemma:young}
   For any vectors $a,b\in \R^d$ the polarization identity $ \langle a,b\rangle = \frac{1}{2} (\|a+b\|^2-\|a\|^2-\|b\|^2)$ holds.
\end{lemma}

    We start by upperbounding $\langle u_t,w_t-w^\ast\rangle$ via leveraging Lemma \ref{lemma:young}
    \begin{equation}\label{thm1:keyeq}
        \begin{aligned}
            &\langle u_t,w_t-w^\ast\rangle = \frac{1}{\eta_t}\langle \eta_t u_t,w_t-w^\ast\rangle\\&=
            \frac{\|w_t-w^\ast\|^2+\eta_t^2\|u_t\|^2-\|w_t-w^\ast-\eta_t u_t\|^2}{2\eta_t}\\&= \frac{\|w_t-w^\ast\|^2+\eta_t^2\|u_t\|^2-\|w_{t+1}-w^\ast\|^2}{2\eta_t}\\&= \frac{D_t^2-D_{t+1}^2}{2\eta_t}+\frac{\eta_t\|u_t\|^2}{2},
        \end{aligned}
    \end{equation}
    where we used SGM's update \eqref{eq:SGM} and defined $D_t^2:= \|w_t-w^\ast\|^2$. Set $\eta_t = \eta = \frac{D}{G\sqrt{T}}$ where $D = D_1$. 
    Summing both sides of \eqref{thm1:keyeq} over $t$ yields
    \begin{equation}\label{thm1:keyeq1}
    \begin{aligned}
                \frac{1}{T}\sum_{t=1}^T  \langle u_t,w_t-w^\ast\rangle &= \frac{D_1^2-D_{T+1}^2}{2\eta T}+  \frac{\eta}{2T}\sum_{t=1}^T \|u_t\|^2 \\&\leq \frac{D_1^2}{2\eta T}+  \frac{\eta G^2}{2} = \frac{DG}{\sqrt{T}},
    \end{aligned}
    \end{equation}
    where we used $D_{T+1}^2 \geq 0$ and the fact that $u_t$ is the gradient of $f$ or $g$ and by the $G$-Lipschitzness assumption its norm is uniformly bounded by $G$.
    
    Let $\A = \{t\in [T]| g(w_t)\leq \epsilon\}$ and $\B = [T]\backslash \A = \{t\in [T]| g(w_t)>\epsilon\}$. By convexity of $f$ and $g$
    \begin{equation}\label{eq:defAB}
    \begin{aligned}
                &\langle u_t,w_t-w^\ast\rangle \geq f(w_t)- f(w^\ast),\qquad \forall t\in\A\\
                &\langle u_t,w_t-w^\ast\rangle \geq g(w_t)- g(w^\ast) \geq g(w_t) >\epsilon,\qquad \forall t\in\B
    \end{aligned}
    \end{equation}
    using the fact that $w^\ast$ is feasible, i.e. $ g(w^\ast)\leq 0$, and that $g(w_t)>\epsilon$ for all $t\in \B$. Consequently,
\begin{equation}\label{eq:sgmconstat1}
    \frac{DG}{\sqrt{T}}  \geq \frac{1}{T} \sum_{t\in \A} f(w_t)-f(w^\ast) +\frac{1}{T} \sum_{t\in \B} g(w_t)-g(w^\ast),.
\end{equation}
Note that when $\epsilon$ is sufficiently large $\A$ is nonempty. Assuming an empty $\A$, we can find the largest ``bad'' $\epsilon$:
\begin{equation}
    \frac{DG}{\sqrt{T}}  \geq \frac{1}{T} \sum_{t\in \B} g(w_t)-g(w^\ast) > \epsilon_{bad}.
\end{equation}
Thus, let us set $\epsilon = (C+1) \frac{DG}{\sqrt{T}} $ for some $C\geq 0$. With this choice, $\A$ is guaranteed to be nonempty.

Now, we consider two cases. Either $\sum_{t\in \A} f(w_t)-f(w^\ast) \leq 0$ which implies by convexity of $f$ and $g$ for $\bar{w} = \frac{1}{|\A|} \sum_{t\in \A} w_t$ we have
\begin{equation}\label{eq:thm1:result1}
  f(\bar{w})-f(w^\ast) \leq 0 <\epsilon,\qquad  g(\bar{w})\leq \epsilon.
\end{equation}
Otherwise, if $\sum_{t\in \A} f(w_t)-f(w^\ast) >0$
    \begin{equation}\label{eq:sgmconstat1new}
        \begin{aligned}
           \frac{DG}{\sqrt{T}}  &\geq \frac{1}{T} \sum_{t\in \A} f(w_t)-f(w^\ast) +\frac{1}{T} \sum_{t\in \B} g(w_t)-g(w^\ast)\\&>\frac{1}{T} \sum_{t\in \A} f(w_t)-f(w^\ast) +\frac{1}{T} \sum_{t\in \B} \epsilon\\&=\frac{|A|}{T}  \frac{1}{|A|}\sum_{t\in \A} f(w_t)-f(w^\ast) +(1-\frac{|A|}{T}) \epsilon\\&\geq\frac{|A|}{T} \big( f(\bar{w})-f(w^\ast) \big)+(1-\frac{|A|}{T}) \epsilon .
        \end{aligned}
    \end{equation}
By rearranging
\begin{equation}
  \frac{|A|}{T}  \Big(f(\bar{w})-f(w^\ast)-\epsilon\Big)  <\frac{DG}{\sqrt{T}} -\epsilon \leq -C\frac{DG}{\sqrt{T}},
\end{equation}
Implying $f(\bar{w})-f(w^\ast)<\epsilon$ and further by convexity of $g$ for $\bar{w} = \frac{1}{|\A|} \sum_{t\in \A} w_t$, we also have $g(\bar{w})\leq\epsilon$. 
\section{Proof of Theorem \ref{thm:sgm2}}
We first state a  well-known Lemma.
\begin{lemma}\label{lemma:smoothness}
    Let $f:\R^d\rightarrow \R$ be a $L$-smooth function and
    assume $\min_x f(x)>-\infty$. Then for all $x,y \in \R^d$ we have 
        \begin{equation}
        \|\nabla f(x)\|^2 \leq 2L (f(x)-\min_x f(x)).
    \end{equation}
\end{lemma}
Recall \eqref{thm1:keyeq} from the proof of Theorem \ref{thm:sgm1}. Here we will bound $\|u_t\|^2$ using Lemma \ref{lemma:smoothness}:
\begin{equation}
    \|u_t\|^2\leq 2L \times \begin{cases}
        f(w_t)-\Tilde{f},\qquad t\in \A\\ g(w_t)-\Tilde{g},\qquad t\in \B.
    \end{cases}
\end{equation}
Note that $\Tilde{f}\leq f(w^\ast)$ and $\Tilde{g}\leq g(w^\ast) \leq 0$ such that $|\Tilde{g}|\geq |g(w^\ast)|$.
We have by convexity of $f$ and $g$ and definition of $\A$ and $\B$ and \eqref{eq:defAB}
    \begin{equation}\label{thm2:keyeq1}
    \begin{aligned}
               \sum_{t\in \A} f(w_t)-f(w^\ast)+\sum_{t\in \B} g(w_t)-g(w^\ast)&\leq \frac{D^2}{2\eta}+ L\eta\sum_{t\in \A} (f(w_t)-\Tilde{f}\pm f(w^\ast))\\&\qquad+ L\eta\sum_{t\in \B} (g(w_t)-\Tilde{g}\pm g(w^\ast)),
    \end{aligned}
    \end{equation}
    and by rearranging
        \begin{equation}\label{thm2:keyeq2}
    \begin{aligned}
               &\Bigg[\sum_{t\in \A} f(w_t)-f(w^\ast)+\sum_{t\in \B} g(w_t)-g(w^\ast)\Bigg](1-\eta L)\\&\leq \frac{D^2}{2\eta}+ L\eta\sum_{t\in \A} (f(w^\ast)-\Tilde{f})+ L\eta\sum_{t\in \B} (g(w^\ast)-\Tilde{g})\\&= \frac{D^2}{2\eta }+ L\eta \Big[|\A|(f(w^\ast)-\Tilde{f})+|\B|(g(w^\ast)-\Tilde{g})\Big]\\&\leq \frac{D^2}{2\eta}+ L\eta T \Delta_{\max}.
    \end{aligned}
    \end{equation}
    Now, recall $ \eta = \min\Big\{\frac{1}{2L}, \sqrt{\frac{D^2}{2L\Delta_{\max}T}}\big\}$
    such that $1-\eta L \geq 1/2$. With this choice we arrive at
    \begin{equation}
    \begin{aligned}
                \frac{1}{T}\sum_{t\in \A} f(w_t)-f(w^\ast)+\frac{1}{T}\sum_{t\in \B} g(w_t)-g(w^\ast)    \leq\frac{D^2}{\eta T}+ 2L\eta \Delta_{\max}.
    \end{aligned}
    \end{equation}
    Note that if $\frac{1}{2L}\leq \sqrt{\frac{D^2}{2L\Delta_{\max}T}}$ such that $\eta = 1/2L$ we have
    \begin{equation}
       2L\eta \Delta_{\max} \leq 2L \Delta_{\max} \sqrt{\frac{D^2}{2L\Delta_{\max}T}} = \sqrt{\frac{2LD^2\Delta_{\max}}{T}},
    \end{equation}
    and thus
        \begin{equation}
        \begin{aligned}
                    \frac{1}{T}\sum_{t\in \A} f(w_t)-f(w^\ast)+\frac{1}{T}\sum_{t\in \B} g(w_t)-g(w^\ast)  \leq\frac{2LD^2}{T}+ \sqrt{\frac{2LD^2\Delta_{\max}}{T}}.
        \end{aligned}
    \end{equation}
    On the other hand, if $\frac{1}{2L}\geq \sqrt{\frac{D^2}{2L\Delta_{\max}T}}$ such that $\eta = \sqrt{\frac{D^2}{2L\Delta_{\max}T}}$ the upper bound becomes $2\sqrt{\frac{2LD^2\Delta_{\max}}{T}}$.
    Thus, in both cases,
            \begin{equation}
            \begin{aligned}
                        \frac{1}{T}\sum_{t\in \A} f(w_t)-f(w^\ast)+\frac{1}{T}\sum_{t\in \B} g(w_t)-g(w^\ast)  \leq\frac{2LD^2}{T}+2\sqrt{\frac{2LD^2\Delta_{\max}}{T}}.
            \end{aligned}
    \end{equation}
    It is evident that we are now back to a setup analogous to the proof of Theorem \ref{thm:sgm1} (i.e., \eqref{eq:sgmconstat1}), but with a different bound. Accordingly, setting 
    \begin{equation}
        \epsilon = \frac{2LD^2}{T}+2\sqrt{\frac{2LD^2\Delta_{\max}}{T}},
    \end{equation}
    and following an identical argument furnishes the proof.
\section{Proof of Theorem \ref{thm:sgm3}}
Recall the update in \eqref{eq:soft-SGM}. Since $0\leq\sigma_\beta(\cdot)\leq 1$ and $f$ and $g$ are $G$-Lipschitz, we have $\|F(w_t)\| \leq G$. Thus, utilizing the analysis in the proof of Theorem \ref{thm:sgm1} and notably \eqref{thm1:keyeq1} and setting $\eta = \frac{D}{G\sqrt{T}}$ we have $\sum_{t=1}^T  \langle F(w_t),w_t-w^\ast\rangle \leq DG\sqrt{T}$.

    Let $\A = \{t\in [T]| g(w_t)< \epsilon\}$ and $\B = [T]\backslash \A = \{t\in [T]| g(w_t)\geq \epsilon\}$. Note that for all $t\in \B$ it holds that $\sigma_\beta(g(w_t)-\epsilon) = 1$ and $g(w_t)-g(w^\ast)\geq \epsilon$. Further, for all $t\in \A$ if $\sigma_\beta(g(w_t)-\epsilon) \geq 0$ it holds that $g(w_t)-g(w^\ast)\geq g(w_t)\geq\epsilon -1/\beta$. With these observations, using convexity of $f$ and $g$ and decomposing the sum over $t$ according to the definitions of $\A$ and $\B$ yields
    \begin{equation}\label{eq:thm3:keyeq2}
        \begin{aligned}
           DG\sqrt{T}&\geq  \sum_{t\in \A} \sigma_\beta(g(w_t)-\epsilon) \big(g(w_t)-g(w^\ast)\big)\\&\quad+\sum_{t\in \A} \big(1-\sigma_\beta(g(w_t)-\epsilon)\big) \big(f(w_t)-f(w^\ast)\big)\\&\quad +\sum_{t\in \B} g(w_t)-g(w^\ast)\\&\geq\sum_{t\in \A}\big(1-\sigma_\beta(g(w_t)-\epsilon)\big) \big(f(w_t)-f(w^\ast)\big)\\&\quad + \epsilon |\B|+ \left(\epsilon-\frac{1}{\beta}\right)\sum_{t\in \A} \sigma_\beta(g(w_t)-\epsilon).
        \end{aligned}
    \end{equation}
Similar to the previous proofs, we first need to find the smallest $\epsilon$ to ensure $\A$ is non-empty. 

To find a lower bound on $\epsilon$ assume $\A$ is empty in \eqref{eq:thm3:keyeq2} and observe that as long as condition $\frac{DG}{\sqrt{T}} < \epsilon $ is met, $\A$ is non-empty. We choose to set $\epsilon = \frac{2DG}{\sqrt{T}}$.

Now like before we consider two cases based on the sign of $\sum_{t\in \A} \big(1-\sigma_\beta(g(w_t)-\epsilon)\big) \big(f(w_t)-f(w^\ast)\big)$. As before, when the sum is non-positive we are done by the definition of $\A$ which implies $0<1-\sigma_\beta(g(w_t)-\epsilon)\leq 1$ and the convexity of $f$ and $g$. 

Assuming the sum is positive, dividing  \eqref{eq:thm3:keyeq2} by $\sum_{t\in \A} \big(1-\sigma_\beta(g(w_t)-\epsilon)\big)$ (which by the definition of $\A$ is strictly positive), using convexity of $f$, and the definition of $\bar{w}$ we have
\begin{equation}
\begin{aligned}
       f(\bar{w})-f(w^\ast) &\leq  \frac{0.5\epsilon T - \epsilon |\B|-(\epsilon-\frac{1}{\beta})\sum_{t\in \A} \sigma_\beta(g(w_t)-\epsilon)}{|\A| - \sum_{t\in \A} \sigma_\beta(g(w_t)-\epsilon)}\\&=\epsilon+\frac{-0.5\epsilon T+\beta^{-1}\sum_{t\in \A} \sigma_\beta(g(w_t)-\epsilon)}{|\A| - \sum_{t\in \A} \sigma_\beta(g(w_t)-\epsilon)},
\end{aligned}
\end{equation}
where we used $|\B| = T-|\A|$.

Let us now find a lower bound on $\beta$ to ensure the second term in the bound is non-positive. Note this is done for simplicity, and as long as the second term is $\O(\epsilon)$, an $\epsilon$-solution can be found. 

Immediate calculations show  the second term in the bound is non-positive when
\begin{equation}
    \beta \geq \frac{2\sum_{t\in \A} \sigma_\beta(g(w_t)-\epsilon)}{\epsilon T}.
\end{equation}
Since $\sum_{t\in \A} \sigma_\beta(g(w_t)-\epsilon) <T$, a sufficient (and highly conservative) condition for all $T\geq 1$ is to set $\beta = 2/\epsilon$. 
Thus, we proved the suboptimality gap result. The feasibility result is immediate given the definition of $\A$ and the convexity of $g$.
\section{Proof of Theorem \ref{thm:sgm5}}
Using Lemma \ref{lemma:young} and the update of SSPPM in \eqref{eq:SSPPM} we have
\begin{equation}
\begin{aligned}
     \langle F(w_{t+1}),w_{t+1}-w^\ast\rangle &= \frac{\|w_{t+1}-w^\ast-\eta F(w_{t+1})\|^2}{2\eta}\\&\qquad\frac{-\|\eta F(w_{t+1})\|^2-\|w_{t+1}-w^\ast\|^2}{2\eta}\\&=\frac{\|w_t-w^\ast\|^2-\|w_{t+1}-w^\ast\|^2}{2\eta}\\&\qquad-\frac{\|w_t-w_{t+1}\|^2}{2\eta}.
\end{aligned}
\end{equation}
Let $w_0 = w_1$. Summing the above result from $t=0$ to $T-1$, change of summand variable  and dropping the negative terms involving $\sum_t\|w_t-w_{t+1}\|^2$ and $\|w_T-w^\ast\|^2$ yields
\begin{equation}
    \sum_{t=1}^T\langle F(w_t),w_t-w^\ast\rangle \leq  \frac{D^2}{2\eta}.
\end{equation}
Thus, we are back to the setting of proof of Theorem \ref{thm:sgm3}. Following identical steps and setting $\epsilon = \frac{D^2}{T\eta}$ and $\beta = \frac{2}{\epsilon}$ furnishes the proof.
\section{Proof of Proposition \ref{prop:lip}}
We will abuse the notation slightly and use $\sigma_\beta(x)$ to refer to $\sigma_\beta(g(x)-\epsilon)$ in this proof. By expanding and regrouping, we have
\begin{equation}
\begin{aligned}
    F(x)& - F(y) 
= (1 - \sigma_\beta(x)) \nabla f(x) + \sigma_\beta(x) \nabla g(x) \\&- (1 - \sigma_\beta(y)) \nabla f(y) - \sigma_\beta(y) \nabla g(y) \\
&= (1 - \sigma_\beta(x)) \left[\nabla f(x) - \nabla f(y)\right] + \sigma_\beta(x) \left[\nabla g(x) - \nabla g(y)\right] \\
&\quad + (\sigma_\beta(y) - \sigma_\beta(x)) \left[\nabla g(y) - \nabla f(y)\right].
\end{aligned}
\end{equation}
Taking norms and applying the triangle inequality yields
\begin{equation}
    \begin{aligned}
\|F(x) - F(y)\|
&\le |1 - \sigma_\beta(x)| \cdot \|\nabla f(x) - \nabla f(y)\| \\&+ |\sigma_\beta(x)| \cdot \|\nabla g(x) - \nabla g(y)\| \\
&\quad + |\sigma_\beta(x) - \sigma_\beta(y)| \cdot \|\nabla g(y) - \nabla f(y)\|.
\end{aligned}
\end{equation}
The first and second term in the equation above are readily bounded by using the smoothness of $f$ and $g$ and the fact that $1\ge \sigma_\beta(x)\ge 0$. The term $\|\nabla g(y) - \nabla f(y)\|$ is also bounded by $2G$ using the Lipschitzness of $f$ and $g$. Now, with our abuse of notation and the non-expansiveness of convex projection
\begin{equation}
\begin{aligned}
        |\sigma_\beta(x) - \sigma_\beta(y)| &\leq |1+\beta (g(x)-\epsilon)-1-\beta (g(y)-\epsilon)|\\&\leq\beta G \|x-y\|.
\end{aligned}
\end{equation}
\section{Proof of Theorem \ref{thm:sgm6}}
Recall the update of SSPPM-E in \eqref{eq:SSPPM-E}. As $h$ is convex, by the first-order optimality condition 
\begin{equation}\label{eq:sppm-update2}
    w_{t+1} = w_t - \eta \nabla h_t(w_{t+1}).
\end{equation}
Now, for each $t$ note that by the definition of $h_t$ we can write
\begin{equation}\label{eq:thm6:main}
    \begin{aligned}
       h_t(w_t) - h_t(w^\ast) &= \underbrace{h_t(w_t) - h_t(w_{t+1})}_{(\text{I)}} + \underbrace{h_t(w_{t+1}) - h_t(w^\ast)}_{\text(II)}
    \end{aligned}
\end{equation}
We can use the $G$-Lipschitzness of $f$ and $g$ as well as the boundedness of $\sigma_\beta$ and the update in \eqref{eq:sppm-update2} to bound the $(I)$ of \eqref{eq:thm6:main} according to 
\begin{equation}
    \begin{aligned}
        &h_t(w_t) - h_t(w_{t+1}) = |h_t(w_t) - h_t(w_{t+1}))| 
        \\& =|[1-\sigma_\beta(g(w_t)-\epsilon)] (f(w_t)-f(w_{t+1}))+ \sigma_\beta(g(w_t)-\epsilon) (g(w_t)-g(w_{t+1}))|
        \\&\leq G\|w_t-w_{t+1}\| \\&= G \|\eta \nabla h_t(w_{t+1})\| \leq \eta G^2.
    \end{aligned}
\end{equation}
Thus, overall $(I)$ is
\begin{equation}
    \begin{aligned} 
        &\sum_{t=1}^T h_t(w_t) - h_t(w_{t+1}) \leq T\eta G^2.
    \end{aligned}
\end{equation}
Next, we will bound $(II)$ on the right hand side of \eqref{eq:thm6:main} using \eqref{eq:sppm-update2} and Lemma \ref{lemma:young}:
\begin{equation}
\begin{aligned}
   &\langle \nabla h_t(w_{t+1}),w_{t+1}-w^\ast\rangle =-\frac{\|w_{t+1}-w^\ast\|^2}{2\eta}\\&\quad-\frac{\|\eta \nabla h_t(w_{t+1})\|^2}{2\eta}+\frac{\|w_{t+1}-w^\ast+\eta \nabla h_t(w_{t+1})\|^2}{2\eta}\\&=\frac{\|w_t-w^\ast\|^2-\|w_{t+1}-w^\ast\|^2-\|w_t-w_{t+1}\|^2}{2\eta}.
\end{aligned}
\end{equation}
Summing the above result from $t=1$ to $T$ and dropping the negative terms involving $\sum_t\|w_t-w_{t+1}\|^2$ and $\|w_{T+1}-w^\ast\|^2$ yields
\begin{equation}
    \sum_{t=1}^T\langle \nabla h_t(w_{t+1}),w_{t+1}-w^\ast\rangle \leq  \frac{D^2}{2\eta}.
\end{equation}
Clearly, by convexity and the definition of $h_t$
\begin{equation}
\begin{aligned}
        \sum_{t=1}^T\langle \nabla h_t(w_{t+1}),w_{t+1}-w^\ast\rangle \geq  \sum_{t=1}^Th_t(w_{t+1}) - h_t(w^\ast).
\end{aligned}
\end{equation}
Thus, we are back to the setting of proof of Theorem \ref{thm:sgm3}. Following identical steps and setting $\eta = \frac{D}{G\sqrt{2T}}$, $\epsilon = \frac{2\sqrt{2}DG}{\sqrt{T}}$, and $\beta = \frac{2}{\epsilon}$ furnishes the proof.
\end{document}